\documentclass[11pt, twoside, leqno]{article}

\usepackage{amsthm}
\usepackage{amssymb}
\usepackage{amsmath}
\usepackage{mathrsfs}

\usepackage{graphics,color}
\usepackage[Symbol]{upgreek}
\usepackage{amscd,epsfig,esint}
\usepackage{verbatim}
\usepackage[english]{babel}

\allowdisplaybreaks 
\def\rr{{\mathbb R}}
\def\rn{{{\rr}^n}}
\def\rtwo{{\mathbb R}^2}

\def\div{{\mathrm{div}}}

\def\loc{{\mathop\mathrm{\,loc\,}}}

\def\ez{\epsilon}

\def\la{\langle}
\def\ra{\rangle}

\newtheorem{thm}{Theorem}
\newtheorem{lem}[thm]{Lemma}

\newtheorem{defn}[thm]{Definition}

\newtheorem{rem}[thm]{Remark}

\begin{document}
\arraycolsep=1pt
\author{} \arraycolsep=1pt
\arraycolsep=1pt

\title{\bf\Large A Note on Transport Equation \\
in Quasiconformally Invariant Spaces \footnotetext{\hspace{-0.35cm}
2010 \emph{Mathematics Subject Classification}. Primary 35F05;
Secondary 35F10.
\endgraf
{\it Key words and phrases}. transport equation, $\operatorname{BMO}$, vector fields,
quasiconformal mapping }}
\author{\it Albert Clop, Renjin Jiang, Joan Mateu \& Joan Orobitg}
\date{}

\maketitle

\begin{center}
\begin{minipage}{10cm}\small
{\noindent{\bf Abstract} In this note, we study the well-posedness
of the Cauchy problem for the transport equation in the $\operatorname{BMO}$ space
and certain Triebel-Lizorkin spaces.}
\end{minipage}

\end{center}

\vspace{0.2cm}

\section{Introduction}

\noindent
In fluid mechanics, the Euler equation
$$\begin{cases}
\dfrac{d v}{dt} + v\cdot \nabla v = 0,\\
\div( v)=0
\end{cases}
$$
describes the motion of an incompressible, inviscid fluid with
velocity $v:[0,T]\times \rn\to\rn$ whose initial state
$v(0,\cdot)=v_0$ is given. When $n=2$, one can reformulate the
system in scalar terms. Namely, one uses the \emph{vorticity}
$\omega:[0,T]\times\rtwo\to \rr$, which is the scalar curl of
$v=(v_1,v_2)$,
$$\omega = \frac{\partial v_2}{\partial x_1}-\frac{\partial v_1}{\partial x_2}.$$
The Biot-Savart law then makes it possible to recover $v(t,\cdot)$ from the vorticity $\omega(t,\cdot)$ by means of the convolution with the complex valued kernel $\frac{i}{2\pi\,\overline{z}}$. Moreover, one obtains for $\omega$ the following equation,
\begin{equation}\label{vorticity}
\frac{d\omega}{dt} + v\cdot \nabla\omega=0,\hspace{1cm}\text{where }v(t,\cdot)=\frac{i}{2\pi\overline{z}}\ast\omega(t,\cdot),
\end{equation}
together with the initial condition $\omega(0,\cdot)=\frac{i}{2\pi
\overline{z}}\ast v(0,\cdot)$. This can be seen as a \emph{scalar}
transport equation for $\omega$, still nonlinear because the
velocity field $v$ depends on the unknown $\omega$. Under the
assumption $\omega_0\in L^1\cap L^\infty$, Yudovich proved global
existence and uniqueness of solutions $\omega\in
L^\infty(0,T;L^\infty)$ (cf. \cite{Yu63},\cite{Yu95},\cite[Chapter 8]{mb02}). In the recent
years, there has been many attempts to understand the case of
unbounded vorticities. Particular attention is devoted to spaces
that stay close to $\operatorname{BMO}$, the space of functions of \emph{bounded
mean oscillation}. This space arises naturally since it contains the
image of $L^\infty$ under any Calder\'on-Zygmund singular integral
operator. Notice that, according to Bourgain-Li \cite{bl15},
\eqref{vorticity} is strongly ill-posed in the borderline space
$W^{1,2}(\mathbb{R}^2)$, while the equation \eqref{vorticity} is not
completely understood in $\operatorname{BMO}$. Recently Bernicot-Keraani
\cite{bk14} extended the well-posedness of \eqref{vorticity} to a
sub-class of $\operatorname{BMO}$, which in particular contains unbounded
vorticities; see also \cite{bh14,cmz13} for further developments.
\\

\noindent Scalar nonlinear transport equations do not only arise
from the Euler equation. Other examples include the surface
quasigeostrophic equation, and the aggregation equation. The general
model is
\begin{equation}\label{nonlineartransport}
\begin{cases}
\dfrac{du}{dt}  + b_u\cdot\nabla u =0\\
u(0,\cdot)=u_0
\end{cases}
\end{equation}
with unknown $u:[0,T]\times \rn\to\rr$.  The nonlinearity comes from the velocity field $b=b_u$,
which may depend on the unknown $u$.
\\

\noindent To study the nonlinear Cauchy problem
\eqref{nonlineartransport}, one of the methods is to first deal with
the corresponding linear problem, i.e., $b$ is independent of $u$.
For example, in the Euler equation, one can first find a suitable
condition on $b$ to solve the linear case and then use the explicit
formula of $b$ in terms of the solution $u$ to play the compactness
argument; see \cite{bk14} for instance.
\\

\noindent Our central problem here is to find suitable conditions on
the vector field $b$ to solve the Cauchy problems for the linear
transport equation with initial value in $\operatorname{BMO}$. In the case of bounded $u_0$, these problems
were successfully treated with the DiPerna-Lions scheme (cf. \cite{dl89})
 and the notion of renormalized solution, as well as the more
recent extensions by Ambrosio (cf. \cite{a04}) in the bounded
variation setting. In both approaches, the starting point is
the classical Cauchy-Lipschitz theory, which allows to write the
solution $u=u(t,x)$ of
\begin{equation}\label{lineartransport}
\begin{cases}
\dfrac{d}{dt}\,u +b\cdot\nabla u =0\\
u(0,\cdot)=u_0
\end{cases}
\end{equation}
as the composition
\begin{equation}\label{composition}
u(t,x)=u_0\circ\phi_t(x)
\end{equation}
where $\phi_t:\rn\to\rn$ is the flow generated by the velocity field $b$,
\begin{equation}\label{flow}
\begin{cases}
\dfrac{d}{dt}\phi_t(x)=-b(\phi_t(x)),\\\phi_0(x)=x,
\end{cases}
\end{equation}
if the vector field $b$ is smooth enough and autonomous. Towards finding explicit solutions
$u\in L^\infty(0,T; \operatorname{BMO})$ of the problem \eqref{lineartransport} for
a given $u_0\in \operatorname{BMO}$, there are two things to be analyzed. First,
describing the class $\cal{Q}$ of homeomorphisms $\phi_t$ under
which \eqref{composition} defines a bounded operator in $\operatorname{BMO}$. Second, describing the class of
velocity fields $b$ such that \eqref{flow} has a solution $\phi_t$ that
falls into $\cal{Q}$. Both questions were analyzed by Reimann \cite{re74,re76} in the 70's.
In the first case (cf. \cite{re74}), quasiconformality was found to be the fundamental notion.
In the second (cf. \cite{re76}), uniform bounds for the anticonformal part of $Db$ were proven to be enough.
\\

The novelty of this work is to apply some known results from
quasiconformal theory to the transport equation with initial data in
quasiconformally invariant spaces.
\\

For a vector field $b:[0,T]\times\rn\to \rn$ being such that $b\in
L^1(0,T;W^{1,1}_\loc)$, let $Db$ be the gradient matrix of $b$ and
$$S_Ab :=\dfrac{1}{2}(Db + Db^{t})-\dfrac{\mathrm{div}\,b}{n}\, I_{n\times n}$$ the anticonformal part of $Db$.
Let us mention that if $S_Ab(t,\cdot)$ is in $L^{\infty}(\rn)$ then $b(t,\cdot)$ is in the Zygmund class \cite{re76}.
Our main result is the following.

\begin{thm}\label{main1}
Let $b:[0,T]\times\rn\to \rn$ be such that $b\in L^1(0,T;W^{1,1}_\loc)$ and
\begin{equation}\label{mainlog}
 \frac{b(t,x)}{1+|x|\log^+|x|}\in L^1(0,T;L^\infty).
\end{equation}

If $S_Ab\in L^1(0,T;L^\infty)$, then for each $u_0\in \operatorname{BMO}$, the
problem \eqref{lineartransport} admits a unique weak solution $u\in
L^\infty(0,T;\operatorname{BMO})$. Moreover, for each $t\in (0,T]$, it holds
$$\|u\|_{L^\infty(0,T;\operatorname{BMO})}\le C(T,b)\|u_0\|_{\operatorname{BMO}}.$$
\end{thm}

\noindent The proof of existence is based
on the fact that \eqref{flow} can be found a unique solution
$\phi_t$ consisting of quasiconformal mappings, which preserve $\operatorname{BMO}$ by
composition. Indeed, it is precisely the assumption $S_Ab \in
L^1(0,T;L^\infty)$ what allows for a classical compactness argument
in $\cal{Q}$. Uniqueness follows as a consequence of renormalization
properties of solutions to the transport equations; see
\cite{dl89,a04}.
\\

\noindent It is a classical fact for harmonic analysts that  $\operatorname{BMO}$
can be identified with the homogeneous Triebel-Lizorkin space
$\dot{F}^0_{\infty,2}$. As it was proven by \cite{kyz11} (see also
\cite{ReRi}), the homogeneous spaces $\dot{F}^\theta_{p,q}$ are
quasiconformally
invariant provided that $\theta p = n$ and $q> \frac{n}{n+\theta}$. As a consequence,
we obtain well-posedness of \eqref{lineartransport} also in these spaces (see Theorem \ref{wp-tran-tls}).
Moreover, well-posedness also holds in the homogeneous Sobolev spaces
$\dot{W}^{1,n}$ (see Theorem \ref{wp-tran-tls}).\\

The paper is organized as follows.  In Section 2 we show that for
vector fields  $b$ satisfying the requirements from Theorem
\ref{main1} the corresponding flow $\phi_t$ from \eqref{flow} is a
quasiconformal mapping for each $t$. The argument is based on Reimann's
approach from \cite{re76}, but we also relax the condition $S_Ab\in
L^\infty(0,T;L^\infty)$ from \cite{re76} to $S_Ab\in
L^1(0,T;L^\infty)$. In Section 3, we prove Theorem \ref{main1}, and
in the last section we address the Cauchy problem for the transport
equation in some Triebel-Lizorkin spaces.

\section{Flows of quasiconformal mappings}
\hskip\parindent In this section, we deal with the flows of
quasiconformal maps. The idea of this section is similar to Reimann
\cite{re76}.
\begin{lem}\label{diff-anti}
If $b:\,\rn\mapsto\rn$ is differentiable at $x$ and
$|S_Ab(x)|<\infty$, then
$$\limsup_{y,z\to 0,\, 0<|z|,|y|}\left|\frac{\la y,(b(x+y)-b(x))\ra}{|y|^2}-\frac{\la z,(b(x+z)-b(x))\ra}{|z|^2}\right|\le 2|S_Ab(x)|.$$
\end{lem}
\begin{proof}
By \cite[Proposition 13]{re76}, it holds that
\begin{eqnarray*}
\limsup_{y\to0}\left|\frac{\la y,(b(x+y)-b(x))\ra}{|y|^2}- \frac{\la
\frac{|y|}{|z|}z,(b(x+\frac{|y|}{|z|}z)-b(x))\ra}{|y|^2}\right|\le
2|S_Ab(x)|.
\end{eqnarray*}
On the other hand, by the differentiability of $b$, we can further
deduce that
\begin{eqnarray*}
&&\lim_{y,z\to 0,\, 0<|z|,|y|}\left|\frac{\la z,(b(x+z)-b(x))\ra}{|z|^2}-
\frac{\la\frac{|y|}{|z|}z,(b(x+\frac{|y|}{|z|}z)-b(x))\ra}{|y|^2}\right|\\
&&=\lim_{y,z\to 0,\, 0<|z|,|y|}\left|\frac{\la z,Db(x)z+o(|z|)\ra}{|z|^2}
-\frac{\la \frac{|y|}{|z|}z,Db(x)\frac{|y|}{|z|}z+o(|y|))\ra}{|y|^2}\right|\\
&&=0.
\end{eqnarray*}
The above two estimates give the desired conclusion.
\end{proof}

\begin{defn}[Distortion]
Let $\phi:\rn\mapsto\rn$ be a homeomorphism. For each $x\in\rn$ and
each $r>0$, define
$$L_{\phi}(x,r)=\sup_{y:\,|y-x|=r}|\phi(y)-\phi(x)|,$$
and
$$\ell_{\phi}(x,r)=\inf_{y:\,|y-x|=r}|\phi(y)-\phi(x)|.$$
We then define the linear distortion function as
$$H_{\phi}(x):=\limsup_{r\to 0} \frac{L_{\phi}(x,r)}{\ell_{\phi}(x,r)}.$$
\end{defn}
A homeomorphism $\phi:\rn\mapsto\rn$ is called a quasiconformal mapping,
if there exists $H>0$ such that  the distortion
$H_{\phi}(x)\le H$ for all $x\in \rn$. Notice that this (metric) definition
coincides with the usual (analytic) definition of $K$-quasiconformal mapping.
Recall that a homeomorphism $\phi:\rn\mapsto\rn$ is called a $K$-quasiconformal mapping, if
$\phi\in W^{1,1}_\loc(\rn)$ with $|D\phi(x)|^n\le K_\phi J_\phi(x)$ for a.e. $x\in\rn$.
Then for any $K$-quasiconformal mapping $\phi$, it holds $K_\phi^{\frac 1{n-1}}\le H_\phi(x)\le K_\phi$ almost everywhere. See the book
\cite{im01} for more information on quasiconformal mappings in $\rn$.

\begin{thm}\label{pri-ptw}
Let $b(t,x):\,[0,T]\times\rn\mapsto \rn$ be a vector field in $L^1(0,T;W^{1,1}_\loc)$
and $b(t,\cdot)\in C^2(\rn)$ for each $t\in [0,T]$.
Assume that $b$ satisfies \eqref{mainlog}
and $S_Ab\in L^1(0,T;L^\infty).$ Then there exists a unique flow of
quasiconformal mappings $\phi_t(x)$ satisfying
$$
\dfrac{d}{dt}\, \phi_t(x)=b(t,\phi_t(x)),\quad \text{ a.e. $t\in [0,T]$,} \qquad \phi_0(x)=x.
$$
Moreover, for each $x\in\rn$ and each $t\in [0,T]$, it holds that
$$H_{\phi_t}(x)\le \exp\left(\int_0^t2|S_Ab(s,\phi_s(x))|\,ds\right).$$
\end{thm}

\begin{proof}
 The existence and uniqueness of a flow $\phi_t(x)$ satisfying
$$\dfrac{d}{dt}\,  \phi_t(x)=b(t,\phi_t(x)) \quad \text{ a.e.  $t\in [0,T]$,}$$
is a classical result; see \cite{ha80}. Moreover, for each $t\in [0,T]$, the flow
$$\phi_t(x)=x+\int_0^tb(s,\phi_s(x))\,ds$$
is a locally Lipschitz homeomorphism of $\rn$ and preserves the class of sets of measure zero.
By \eqref{mainlog} one has
$$
|\phi_t(x)| \le |x| +\int_{0}^{t} \left\| \frac{b(s,\cdot)}{1+|\cdot|\log^+|\cdot|} \right\|_{\infty}
(1+|\phi_s(x)| |\log^+|\phi_s(x)|) \, ds.
$$
Then, using a Gronwall type inequality  due to I. Bihari (see \cite[p. 3]{dr03})
one gets
\begin{equation}\label{bounduni}
 |\phi_t(x)| \le C(R,b)\quad \text{for all } (t,x)\in [0,T]\times B(0,R),
\end{equation}
where $C(R,b)$ is a constant depending on the radius $R$ and $\displaystyle \int_0^{T} \left\| \frac{b(s,\cdot)}{1+|\cdot|\log^+|\cdot|} \right\|_{\infty} \, ds$,
that is, $\phi_t$ maps bounded sets into bounded sets in finite time.

Let $x\in \rn$ and $t\in [0,T]$ be fixed. For each $y,z\in B(0,1)$,
$|y|=|z|\neq 0$, define
$$A(t,x)=\phi_t(x+y)-\phi_t(x)$$
$$B(t,x)=\phi_t(x+z)-\phi_t(x),$$
$$D(t,x)=b(t,\phi_t(x+y))-b(t,\phi_t(x)),$$
$$E(t,x)=b(t,\phi_t(x+z))-b(t,\phi_t(x)).$$
and set
$$H_{y,z}(t,x)=\frac{|A(t,x)|}{|B(t,x)|}.$$
Because, for each $t\in [0,T]$, $\phi_t$ is a homeomorphism of $\rn$,  the quantity $H_{y,z}(t,x)$ is well defined.
It is clear from the definiton that $\log H_{y,z}(t,x)$ as function of $t$ is absolutely continuous on $[0,T]$.
For $|s|$ small enough such that $t+s\in [0,T]$, one has
\begin{eqnarray*}
H_{y,z}(t+s,x)-H_{y,z}(t,x)&&=\frac{|A(t+s,x)|}{|B(t+s,x)|}-\frac{|A(t,x)|}{|B(t,x)|}\\
&&=\frac{|A(t,x)|}{|B(t+s,x)|}\left\{\frac{|A(t+s,x)|}{|A(t,x)|}-\frac{|B(t+s,x)|}{|B(t,x)|}\right\}\\
&&=\frac{|A(t,x)|}{|B(t+s,x)|}\frac{\left\{\frac{|A(t+s,x)|^2}{|A(t,x)|^2}-\frac{|B(t+s,x)|^2}{|B(t,x)|^2}\right\}}
{\left\{\frac{|A(t+s,x)|}{|A(t,x)|}+\frac{|B(t+s,x)|}{|B(t,x)|}\right\}},
\end{eqnarray*}
and therefore,
\begin{eqnarray*}
\frac{H_{y,z}(t+s,x)}{H_{y,z}(t,x)}&&=\frac{|B(t,x)|}{|B(t+s,x)|}\frac{\left\{\frac{|A(t+s,x)|^2}{|A(t,x)|^2}-\frac{|B(t+s,x)|^2}{|B(t,x)|^2}\right\}}
{\left\{\frac{|A(t+s,x)|}{|A(t,x)|}+\frac{|B(t+s,x)|}{|B(t,x)|}\right\}}+1.
\end{eqnarray*}
Using that,
$$A(t+s,x)=A(t,x)+\int_t^{t+s}D(r,x)\,dr$$
$$B(t+s,x)=B(t,x)+\int_t^{t+s}E(r,x)\,dr,$$
we can conclude that for a.e. $t\in [0,T]$ it holds
\begin{eqnarray}\label{est-dilation}
\frac{d \log H_{y,z}(t,x)}{\,dt}&&=\lim_{s\to
0}\frac 1s\log\left(\frac{H_{y,z}(t+s,x)}{H_{y,z}(t,x)}\right)\nonumber\\
&&=
\lim_{s\to
0}\frac 1s\log\left(\frac{|B(t,x)|}{|B(t+s,x)|}\frac{\left\{\frac{|A(t+s,x)|^2}{|A(t,x)|^2}-\frac{|B(t+s,x)|^2}{|B(t,x)|^2}\right\}}
{\left\{\frac{|A(t+s,x)|}{|A(t,x)|}+\frac{|B(t+s,x)|}{|B(t,x)|}\right\}}+1\right)\nonumber\\
&&=\frac{\la A(t,x),D(t,x)\ra}{|A(t,x)|^2}-\frac{\la B(t,x),E(t,x)\ra}{|B(t,x)|^2}.
\end{eqnarray}

By the estimate \eqref{est-dilation}, we see that
\begin{eqnarray*}
 H_{y,z}(t,x)\le \exp\left\{\int_0^t\left|\frac{\la A(s,x),D(s,x)\ra}{|A(s,x)|^2}
 -\frac{\la B(s,x),E(s,x)\ra}{|B(s,x)|^2}\right|\,ds\right\}.
\end{eqnarray*}
Now, since  $\phi_t$ is locally Lipschitz continuous ($b(t,\cdot)\in C^2(\rn)$ for each $t\in [0,T]$),
we can apply Lemma \ref{diff-anti} to obtain
\begin{eqnarray*}
\limsup_{|y|=|z|\to 0} H_{y,z}(t,x)&&\le \limsup_{|y|=|z|\to 0}
\exp\left\{\int_0^t\left|\frac{\la A(s,x),D(s,x)\ra}{|A(s,x)|^2}-\frac{\la B(s,x),E(s,x)\ra}{|B(s,x)|^2}\right|\,ds\right\}\\
&&\le \exp\left\{\int_0^t \limsup_{|y|=|z|\to 0}\left|\frac{\la
A(s,x),D(s,x)\ra}{|A(s,x)|^2}
-\frac{\la B(s,x),E(s,x)\ra }{|B(s,x)|^2}\right|\,ds\right\}\\
&&\le  \exp\left\{\int_0^t \left|2S_Ab(s,\phi_s(x))\right|\,ds\right\},
\end{eqnarray*}
for all $x\in \rn$ and all $t\in [0,T]$. The proof is completed.
\end{proof}

\begin{thm}\label{qcfl}
Let $b(t,x):\,[0,T]\times\rn\mapsto \rn$ be a vector field in $L^1(0,T;W^{1,1}_\loc)$. Assume \eqref{mainlog}
and $S_Ab\in L^1(0,T;L^\infty).$ Then there exists a unique flow of
quasiconformal mappings  $\phi_t(x)$ satisfying
$$\dfrac{d}{dt} \, \phi_t(x) =b(t,\phi_t(x)),\quad \text{ a.e.  $t\in [0,T]$,} \qquad \phi_0(x)=x.$$
Moreover, for a.e. $x\in\rn$ and each $t\in [0,T]$, it holds that
$$K_{\phi_t} \le \exp\left((n-1) \int_0^t2\|S_Ab(s,\cdot)\|_{L^\infty}\,ds\right).$$
\end{thm}
\begin{proof}
Let $\rho\in C^\infty_c(B(0,1))$ be a non-negative smooth function
that satisfies $\int_\rn \rho(x)\,dx=1$. For each
$\ez>0$ let $\rho_\ez(x)=\rho(x/\ez)/\ez^n$  and
$$b_\ez(t,x):=\int_\rn b(t,x-y)\rho_\ez(y)\,dy.$$
Then for each $\ez>0$, $b_\ez$ satisfies the requirements from Theorem \ref{pri-ptw},
and therefore there exists a unique flow of quasiconformal maps $\phi_{t,\ez}$ that satisfies
$$\dfrac{d}{dt} \, \phi_{t,\ez}(x) =b_\ez(t,\phi_{t,\ez}(x)), \quad \text{ a.e.  $t\in [0,T]$,}\qquad \phi_{0,\ez}(x) = x .$$
Moreover, since
\begin{eqnarray*}
\left|S_Ab_\ez(s,\phi_{s,\ez}(x))\right|&&\le
\int_\rn
\left|S_Ab(s,\phi_{s,\ez}(x)-y)\right|\rho_\ez(y)\,dy\le \|S_Ab(s,\cdot)\|_{L^\infty},
\end{eqnarray*}
we have that the linear distortion function $H_{\phi_{t,\ez}}$ of
$\phi_{t,\ez}$ satisfies
\begin{eqnarray}\label{boundQC}
H_{\phi_{t,\ez}}(x)&&\le
\exp\left(\int_0^t2|S_Ab_\ez(s,\phi_{s,\ez}(x))|\,ds\right)\le \exp\left(\int_{0}^{t}2\|S_Ab(s,\cdot)\|_{L^\infty}\,ds\right).
\end{eqnarray}
Notice that  from \eqref{mainlog} and the argument used to obtain \eqref{bounduni} we have that
for any bounded set $U\subset\rn$, $\phi_{t,\ez}(U)$ is uniformly bounded in $\rn$ for any $t\in [0,T]$ and any $\ez<1$.
That is, $\phi_{t,\ez}(U)\subset B(0,R)$ where the radius $R$ depends on $U$ and $b$.
Moreover, by \eqref{boundQC}, for all $t\in [0,T]$ and $\ez <1$, the map $\phi_{t,\ez}$ is $K$-quasiconformal with
$$K_{\phi_{t,\ez}} \le (H_{\phi_{t,\ez}}(x))^{n-1} \le \exp\left((n-1) \int_{0}^{t}2\|S_Ab(s,\cdot)\|_{L^\infty}\,ds\right).
$$
So, the family $\{ \phi_{t,\ez}\}_{0<\ez<1}$ is locally equicontinuous in the spatial direction. On the other hand, if $x\in U$
$$
\left|\phi_{t,\ez}(x)-\phi_{s,\ez}(x)\right|\le C(U,T)\int_s^t\left\|\frac{b(r,\cdot)}{1+|\cdot|\log^+|\cdot|}\right\|_{\infty} \,dr.
$$
Therefore, we can conclude that $\phi_{t,\ez}(x)$ is locally uniformly bounded and equicontinuous  in $[0,T]\times \rn$.
Applying the Arzel\`a-Ascoli theorem, we achieve  that $\phi_{t,\ez}$ converges to some $\phi_t$ locally uniformly up
to a subsequence. To get
\begin{equation}\label{edo1}
\phi_t(x)=x+\int_0^t b(s,\phi_s(x))\,ds
\end{equation}
from
$$\phi_{t,\ez}(x)=x+\int_0^tb_{\ez}(s,\phi_{s,\ez}(x))\,ds,
$$
it is enough to  prove that for each $x\in \rn$
$$
\int_0^tb_{\ez}(s,\phi_{s,\ez}(x))\,ds \longrightarrow \int_0^tb(s,\phi_s(x))\,ds \qquad \text{as } \ez\to 0.
$$
Thus, we split
\begin{eqnarray*}
\int_0^t |b_{\ez}(s,\phi_{s,\ez}(x)) -  b(s,\phi_s(x)) | \,ds && \le  \int_0^t |b_{\ez}(s,\phi_{s,\ez}(x)) -  b_{\ez}(s,\phi_s(x)) | \,ds \\
&& + \int_0^t |b_{\ez}(s,\phi_s (x)) -  b(s,\phi_s(x)) | \,ds := I + II.
\end{eqnarray*}
Recall that $\| S_A b_{\ez}(s, \cdot)\|_{\infty} \le \| S_A b (s, \cdot)\|_{\infty}<\infty $ a.e.. So, as in \cite{re76} , for a.e.
$s\in [0,1]$ $b_{\ez}(s, \cdot )$ belongs to the Zygmund class. Then
$$
I \le C \int_0^t  \| S_A b (s, \cdot)\|_{\infty}\, |\phi_{s,\ez} (x) - \phi_s (x)| \left | \log |\phi_{s,\ez} (x) - \phi_s (x)| \right |\, ds
$$
which tends to 0 because $\displaystyle\sup_{s\in [0,T]} |\phi_{s,\ez} (x) - \phi_s (x)| \to 0$ as $\ez\to 0$. The dominated convergence theorem gives $II \to 0$ as $\ez \to 0$, because $b_{\ez}\longrightarrow b$ a.e. in $[0,T] \times \rn$ and
$$
| b_{\ez}(s,\phi_s (x) | \le C(x,T) \left\|\frac{b(s,\cdot)}{1+|\cdot|\log^+|\cdot|}\right\|_{\infty}\in L^1 ([0,T]).
$$

Equivalently to \eqref{edo1} we obtained
$$\dfrac{d}{dt} \, \phi_t(x) =b(t,\phi_t(x)),\quad \text{ a.e.  $t\in [0,T]$,} \qquad \phi_0(x)=x.$$

The uniqueness of the flow follows as a consequence that for a.e.
$t$ the vector field $b(t,\cdot)$ is in the Zygmund class and so it satisfies a quasi-Lipschitz condition (see \cite[Theorem 1.5.1]{al93}).

By the fact that a uniform limit of $K$-quasiconformal mappings is a $K$-quasiconformal mapping or a
constant (which can not happen here since $\phi_t$ satisfies the above ODE), we know that $\phi_t$ is a quasiconformal mapping with
$$K_{\phi_{t}}\le \exp\left((n-1)\int_0^{t}2\|S_Ab(s,\cdot)\|_{L^\infty}\,ds\right).
$$
\end{proof}

In order to use the inverse of the flows $\phi_t$ for the vector field $b$ being non-autonomous,we need to introduce general definitions of the
flows.
\begin{defn}
Let $b: [0,T]\times \rn\to\rn$ be a Borel vector field, and $\phi_{s,t},\,\tilde \phi_{s,t}: [0,T ]\times [0,T]\times  \rn\to\rn$ be Borel maps.

We say that $\phi_{s,t}$ is a forward flow associated to $b$ if for each $s\in [0,T]$ and almost every $x\in\rn$  the map
$t\mapsto\,|b(t, \phi_{s,t}(x))|$ belongs to $L^1(s,T )$ and
$$\phi_{s,t}(x) = x +\int_s^t b(r,\phi_{s,r}(x))\,dr.$$
We say that $\tilde \phi_{s,t}(x)$ is a backward flow associated to $b$ if for each $t\in [0,T]$ and almost every $x\in\rn$  the map
$s\mapsto\,|b(s, \tilde \phi_{s,t}(x))|$ belongs to $L^1(0,t)$ and
$$\tilde \phi_{s,t}(x) = x -\int_s^t b(r,\tilde \phi_{r,t}(x))\,dr.$$
\end{defn}
The role of $\tilde \phi_{s,t}$ is that $\tilde \phi_{s,t}$ is the inverse of $\phi_{s,t}$, i.e., $\tilde \phi_{s,t}(\phi_{s,t}(x))=x$; see for instance \cite[Theorem 2.1]{cc05}. The following result is a time dependent version of Theorem \ref{qcfl}.

\begin{thm}\label{exist-qcf}
Let $b(t,x):\,[0,T]\times\rn\mapsto \rn$ be a vector field in $L^1(0,T;W^{1,1}_\loc)$. Assume \eqref{mainlog}
and $S_Ab\in L^1(0,T;L^\infty).$ Then there exists a unique forward flow of
quasiconformal mappings  $\phi_{s,t}(x)$ satisfying
$$\phi_{s,t}(x) = x +\int_s^t b(r,\phi_{s,r}(x))\,dr.$$
and a unique backward flow  of quasiconformal mappings  $\tilde \phi_{s,t}(x)$ satisfying
$$\tilde \phi_{s,t}(x) = x -\int_s^t b(r,\tilde \phi_{r,t}(x))\,dr.$$
Moreover, for a.e. $x\in\rn$ and $0\le s\le t\le T$, it holds that
$$K_{\phi_{s,t}} \le \exp\left((n-1) \int_s^t2\|S_Ab(r,\cdot)\|_{L^\infty}\,dr\right),$$
$$K_{\tilde \phi_{s,t}} \le \exp\left((n-1) \int_s^t2\|S_Ab(r,\cdot)\|_{L^\infty}\,dr\right).$$
\end{thm}

\section{Transport equation in $\operatorname{BMO}$}
\hskip\parindent In this section, we apply the theory of flows of
quasiconformal mappings  to the transport equation
\eqref{lineartransport} with initial value in $\operatorname{BMO}$. Recall that, a
locally integrable function $f$ is in the space $\operatorname{BMO}$, if
$$\|f\|_{\operatorname{BMO}}:=\sup_{B}\frac 1{|B|}\int_B|f-f_B|\,dx<\infty,$$
where $f_B$ denotes $\frac 1{|B|}\int_Bf\,dx$ and the supremum is taken over all open balls. In \cite{re74},
Reimann proved that

\begin{thm}\label{qc-bmo}
The $\operatorname{BMO}$ space is invariant under quasiconformal mappings of $\rn$.
Precisely, for any $K$-quasiconformal mapping $\phi$, there exists
$C=C(K,n)$ such that for any $f$ in $\operatorname{BMO}$, it holds
$$\|f\circ \phi\|_{\operatorname{BMO}}\le C\|f\|_{\operatorname{BMO}}.$$
\end{thm}

A function $u\in L^1(0,T;L^1_{\loc})$ is called a \textit{weak
solution} to \eqref{lineartransport} if for each $\varphi\in
C_c^\infty([0,T]\times\rr^n)$ with compact support in  $[0,T)\times
\rr^n$ it holds that
\begin{equation*}
\int_0^T\int_{\rr^n} u\,\dfrac{\,d \varphi}{\,d t}\,dx\,dt+
\int_{\rr^n} u_0 \,\varphi(0,\cdot)\,dx+\int_0^T\int_{\rr^n} u\,\div
(b\,\varphi)\,dx\,dt=0.
\end{equation*}

\begin{proof}[Proof of Theorem \ref{main1}] Let us first prove the existence. By Theorem \ref{exist-qcf},
 we know that there exists a unique forward flow $\phi_{s,t}(x)$ and a unique backward flow
 $\tilde \phi_{s,t}(x)$ satisfying
 $$\phi_{s,t}(x) = x +\int_s^t b(r,\phi_{s,r}(x))\,dr,$$
$$\tilde \phi_{s,t}(x) = x -\int_s^t b(r,\tilde \phi_{r,t}(x))\,dr.$$
Moreover, $\phi_{s,t}(x)$ and $\tilde \phi_{s,t}(x)$ are flows of quasi-conformal mappings.

Let $u_0\in \operatorname{BMO}$ and $u(x,t):=u_0(\tilde \phi_{0,t}(x))$ for each
$t\in [0,T]$. Since for each $t$, $\tilde\phi_{0,t}$ preserves zeros sets  of $\mathbb{R}^n$, $u(x,t)$ is well defined.
We deduce from Theorem \ref{exist-qcf} that for each fixed $t\in [0,T]$, $\tilde\phi_{0,t}$ is a
$K$-quasiconformal mapping with
$$K_{\tilde \phi_{0,t}}\le \exp\left((n-1) \int_0^t2\|S_Ab(s,\cdot)\|_{L^\infty}\,ds\right).$$
This, together with Theorem \ref{qc-bmo}, implies that $u(x,t)\in
L^\infty(0,T;\operatorname{BMO})$ with
$$\|u\|_{L^\infty(0,T;\operatorname{BMO})}\le C(T,b)\|u_0\|_{\operatorname{BMO}}.$$

Let us next show that $u$ is a weak solution to
\eqref{lineartransport}.
 Choose an arbitrary
$\varphi\in C_c^\infty(\rr^n)$ and $\psi\in C^\infty_c([0,T))$.
By the fact $\phi_{0,t}(\tilde \phi_{0,t}(x))=\tilde \phi_{0,t}(\phi_{0,t}(x))=x$,
we have
\begin{eqnarray*}
&&\int_0^T\int_{\rr^n} u\,\dfrac{\,d (\varphi\psi)}{\,d
t}\,dx\,dt\\
&&\quad=\int_0^T\int_{\rr^n} u_0(\tilde\phi_{0,t}(x))\,\dfrac{\,d \psi(t)}{\,d t}\varphi(x)\,dx\,dt\\
&&\quad=\int_0^T\int_{\rr^n} u_0(x)\,\dfrac{\,d \psi(t)}{\,d t}\varphi(\phi_{0,t}(x))J_{\phi_{0,t}}(x)\,dx\,dt\\
&&\quad=-\int_{\rr^n} u_0(x)\,\psi(0)\varphi(x)\,dx
-\int_0^T\int_{\rr^n} u_0(x)\psi(t)\,\dfrac{\,d
(\varphi(\phi_{0,t}(x))J_{\phi_{0,t}}(x))}{\,d t}\,dx\,dt.
\end{eqnarray*}
On the other hand, it holds
\begin{eqnarray*}
&&\int_0^T\int_{\rr^n} u\,\div
(b\,\varphi\psi)\,dx\,dt\\
&&\quad=\int_0^T\int_{\rr^n} u_0(\tilde \phi_{0,t}(x))\psi(t)
\,\div (b(t,x)\,\varphi(x))\,dx\,dt\\
&&\quad=\int_0^T\int_{\rr^n} u_0(x)\psi(t) \,\div
(b(t,z)\,\varphi(z))|_{z=\phi_{0,t}(x)}J_{\phi_{0,t}}(x)\,dx\,dt.
\end{eqnarray*}
By noticing that $\dfrac{d}{dt}\,J_{\phi_{0,t}}(x) = \div (b(t,\phi_{0,t}(x)) \,J_{\phi_{0,t}}(x)$ (see \cite[Proposition 3.6]{bhs08} and  \cite{a04,re76}),
one has
\begin{eqnarray*}
&&\dfrac{d}{dt} \, (\varphi(\phi_{0,t}(x))J_{\phi_{0,t}}(x)) \\
&&\quad=\nabla \varphi(\phi_{0,t}(x))\cdot
b(t,\phi_{0,t}(x))J_{\phi_{0,t}}(x)+\varphi(\phi_{0,t}(x))
\div\,b(t,\phi_{0,t}(x))J_{\phi_{0,t}}(x))\\
&&\quad=\,\div (b(t,z)\,\varphi(z))|_{z=\phi_{0,t}(x)}J_{\phi_{0,t}}(x).
\end{eqnarray*}
We can conclude that
\begin{eqnarray*}
\int_0^T\int_{\rr^n} u\,\dfrac{\,d (\varphi\phi)}{\,d
t}\,dx\,dt+\int_{\rr^n} u_0(x)\,\psi(0)\varphi(x)\,dx
+\int_0^T\int_{\rr^n} u\,\div (b\,\varphi\psi)\,dx\,dt,=0,
\end{eqnarray*}
which implies that $u(x,t)=u_0(\tilde \phi_{0,t}(x))$ is a weak solution to
\eqref{lineartransport}.

Let us prove the uniqueness. Let $u\in L^\infty(0,T;\operatorname{BMO})$ be a
solution of the transport equation with initial value $u_0=0$.
Notice that for $b$ with $S_Ab \in L^1(0,T;L^\infty)$, we have
$Db\in L^\infty(0,T;L^q_\loc)$ for each finite $q$; see
\cite[p.262]{re76}. Letting $0\le \rho\in C^\infty_c(\rn)$, and
$\rho_\epsilon=\epsilon^{-n}\rho(\cdot/\epsilon)$ for each
$\epsilon>0$, we conclude by \cite[Theorem 2.1, Lemma 2.1]{dl89}
that
$$\dfrac{\,d u_\epsilon}{\,d t}+b\cdot \nabla u_\epsilon=r_\epsilon,$$
where $u_\epsilon=u\ast\rho_\epsilon$, and $r_\epsilon\to 0$ in
$L^1(0,T;L^1_\loc)$ as $\epsilon\to 0$. Therefore, using the renormalization property of transport equation,
one has that for each
$\beta\in C^1(\rr)$ with $\beta(0)=0$ and $\beta'\in L^\infty$, it
holds
$$\dfrac{\,d \beta(u)}{\,d t}+b\cdot \nabla \beta(u)=0,$$
i.e., $\beta(u)$ is a solution of the transport equation with initial value $\beta(u_0)=0$.
For each $M>0$, let $\beta_M(t)=|t|\wedge M$ be a Lipschitz function on $\rr$.
A further approximation argument would give us that,
$\beta_M(u)=|u|\wedge M$ is a solution of the transport equation with initial value $\beta(u_0)=0$.

At this point, applying the well-posedness of the transport equation
in $L^\infty$ (see e.g. \cite[Theorem 2.2]{cjmo15}) gives us that
$\beta_M(u)=0$ for each $M>0$. Letting $M\to \infty$, we conclude
that $u=0$.
\end{proof}

\begin{rem}\rm
By the Banach-Alaoglu theorem, one has that the solution $u$ found in Theorem \ref{main1} is
continuous in time with respect to the weak-$\ast$ topology of $\operatorname{BMO}$.

If instead of $\operatorname{BMO}$ one considers the vanishing mean oscillation space ($\operatorname{VMO}$), which is defined as the closure of
compactly supported smooth functions with the $\operatorname{BMO}$ norm, then under the assumptions of Theorem \ref{main1},
for each $u_0\in \operatorname{VMO}$, one can find a unique solution $u$ in $L^\infty(0,T;\operatorname{VMO})$. Moreover, the solution $u$
is continuous in time with respect to the norm topology of $\operatorname{VMO}$, that is, $u\in C(0,T;\operatorname{VMO})$. Indeed, since
the solution $u$ is given { by $u_0( \tilde \phi_{0,t})$ as in the proof of Theorem \ref{main1}},
for each compactly supported smooth function $u_0$, it is easy to see that {$u_0(\tilde \phi_{0,s})\to u_0(\tilde \phi_{0,t})$}
uniformly as $s\to t$ and therefore,
$$\left\|u(s,\cdot)-u(t,\cdot)\right\|_{\operatorname{BMO}}\le 2\|u(s,\cdot)-u(t,\cdot)\|_{L^\infty}\to 0,\quad \mathrm{as}\ s\to t.$$
An density argument gives the desired conclusion for any initial value in $\operatorname{VMO}$.
\end{rem}


\section{Transport equation in Triebel-Lizorkin spaces}

In this section, we show that the same conclusion of Theorem
\ref{main1} holds with $\operatorname{BMO}$ replaced by certain Triebel-Lizorkin
spaces.

The following result was proved in Koskela-Yang-Zhou \cite{kyz11}.
We refer the reader to \cite{kyz11} for precise definitions of the
Triebel-Lizorkin spaces.
\begin{thm}
Let $n\ge 2$, $s\in (0,1)$ and $q\in (n/(n + s),\infty]$ Then
$\dot{F}^s_{n/s,q}(\rn)$ is invariant under quasiconformal mappings
of $\rn$.
\end{thm}

Applying the above theorem and the well-known fact that
$\dot{W}^{1,n}(\rn)$ is also quasiconformally invariant (cf.
\cite{kyz11}), similar to the proof of Theorem \ref{main1} we can
conclude the following result, whose proof will be omitted.

\begin{thm}\label{wp-tran-tls}
Let $b(t,x):\,[0,T]\times\rn\mapsto \rn$ be a vector field in $L^1(0,T;W^{1,1}_\loc)$. Assume that $b$ satisfies
\eqref{mainlog}
and $S_Ab\in L^1(0,T;L^\infty).$ Then
\begin{enumerate}
 \item[(i)] for each $u_0\in \dot{F}^s_{n/s,q}(\rn)$, $s\in (0, 1)$ and
$q\in (n/(n + s),\infty]$, there exists a unique solution $u\in
L^\infty(0,T;\dot{F}^s_{n/s,q}(\rn))$ of \eqref{lineartransport}.
\item[(ii)] for each $u_0\in \dot{W}^{1,n}(\rn)$, there exists a unique solution
$u\in L^\infty(0,T;\dot{W}^{1,n}(\rn))$ of \eqref{lineartransport}.
\end{enumerate}
\end{thm}

\subsection*{Acknowledgement}
\hskip\parindent A. Clop, J. Mateu and J. Orobitg were partially
supported by Generalitat de Catalunya (2014SGR75) and Ministerio de
Economia y Competitividad (MTM2013-44699). A. Clop was partially supported by
the Programa Ram\'on y Cajal (Spain).
R. Jiang was partially supported by National Natural Science Foundation of China (NSFC
11301029). All authors were partially supported by Marie Curie Initial Training Network MAnET (FP7-607647).

\noindent Albert Clop, Joan Mateu and Joan Orobitg

\vspace{0.1cm} \noindent Department of Mathematics, Faculty of
Sciences, Universitat Aut\`onoma de Barcelona, 08193 Bellaterra
(Barcelona), Catalunya

\vspace{0.3cm}

\noindent Renjin Jiang

\vspace{0.1cm}
\noindent
School of Mathematical Sciences, Beijing Normal University, Laboratory of Mathematics and Complex Systems,
Ministry of Education, 100875, Beijing, CHINA

and

\noindent Department of Mathematics, Faculty of Sciences,
Universitat Aut\`onoma de Barcelona, 08193 Bellaterra (Barcelona),  Catalunya

\vspace{0.2cm}
\noindent{\it E-mail addresses}:\\
\texttt{albertcp@mat.uab.cat}\\
\texttt{rejiang@bnu.edu.cn}\\
\texttt{mateu@mat.uab.cat}\\
\texttt{orobitg@mat.uab.cat}
\end{document}